\documentclass[10pt]{article}

\usepackage[cp1252]{inputenc}
\usepackage[T1]{fontenc}
\usepackage{lmodern}
\usepackage[a4paper]{geometry}
\usepackage{amsmath, amssymb}
\usepackage[english]{babel}

\usepackage{enumerate}
\usepackage{amsthm}
\usepackage{amscd}
\numberwithin{equation}{section}

\usepackage[colorlinks=true]{hyperref}
\hypersetup{pdftitle={Class formula},pdfauthor={Florent Demeslay}}

\usepackage{xcolor}

\theoremstyle{plain}
\newtheorem{thm}{Theorem}[section]
\newtheorem{lem}[thm]{Lemma}
\newtheorem{prop}[thm]{Proposition}

\theoremstyle{definition}
\newtheorem*{defi}{Definition}
\newtheorem*{exe}{Example}

\theoremstyle{remark}

\newcommand{\ie}{\emph{i.e.} }

\newcommand{\fq}{\mathbb F_q}
\newcommand{\kinf}{K_\infty}
\newcommand{\linf}{L_\infty}
\newcommand{\ksinf}{K_{s,\infty}}
\newcommand{\lsinf}{L_{s,\infty}}
\newcommand{\cinf}{\mathbb C_\infty}
\newcommand{\csinf}{\mathbb C_{s,\infty}}
\newcommand{\rls}{R_{L,s}}
\newcommand{\rss}{R_{S,s}}
\newcommand{\lsp}{L_{s,\mathfrak p}}
\newcommand{\lss}{L_{s,S}}
\newcommand{\rsps}{R_{S \cup \{\mathfrak p\},s}}
\newcommand{\m}{\mathfrak m}
\newcommand{\p}{\mathfrak p}

\newcommand{\osp}{\mathcal O_{s,\p}}
\newcommand{\ol}{\mathcal O_L}
\newcommand{\os}{\mathcal O_S}
\newcommand{\calus}{\mathcal U_s}

\DeclareMathOperator{\im}{Im}

\DeclareMathOperator{\fitt}{Fitt}

\DeclareMathOperator{\expe}{exp_E}
\DeclareMathOperator{\expei}{exp_E^{-1}}
\DeclareMathOperator{\expa}{exp_{\alpha}}

\DeclareMathOperator{\expachi}{exp_{\alpha(\chi)}}
\DeclareMathOperator{\expachii}{exp_{\alpha(\chi)}^{-1}}

\DeclareMathOperator{\expcn}{exp_{C^{\otimes n}}}
\DeclareMathOperator{\expcni}{exp_{C^{\otimes n}}^{-1}}
\DeclareMathOperator{\cn}{C^{\otimes n}}
\DeclareMathOperator{\cnt}{C^{\otimes n}_\theta}
\DeclareMathOperator{\logcn}{log_{C^{\otimes n}}}

\DeclareMathOperator{\enm}{End}

\DeclareMathOperator{\lie}{Lie}
\DeclareMathOperator{\gln}{GL_n}
\DeclareMathOperator{\gl}{GL}

\let\hom\relax
\DeclareMathOperator{\hom}{Hom}

\title{A class formula for $L$-series in positive characteristic}
\author{Florent Demeslay}
\date{January, 2015}

\begin{document}

\maketitle

\begin{abstract}
We prove a formula for special $L$-values of Anderson modules, analogue in positive characteristic of the class number formula.
We apply this result to two kinds of $L$-series.
\end{abstract}

\section{Introduction} \label{sec-intro}

Let $\fq$ be a finite field with $q$ elements and $\theta$ an indeterminate over $\fq$.
We denote by $A$ the polynomial ring $\fq[\theta]$ and by $K$ the fraction field of $A$.
For a finite $A$-module $M$, we denote by $[M]_A$ the monic generator of the Fitting ideal of $M$.
The \emph{Carlitz zeta value} at a positive integer $n$ is defined as
\[ \zeta_A(n) := \sum\limits_{a \in A_+} \frac{1}{a^n} \in \kinf := \fq((\theta^{-1})), \]
where $A_+$ is the set of monic polynomials of $A$.

The \emph{Carlitz module} $C$ is the functor that associates to an $A$-algebra $B$ the $A$-module $C(B)$ whose underlying $\fq$-vector space is $B$ and whose $A$-module structure is given by the homomorphism of $\fq$-algebras
\[ \begin{array}{cccc}
    \varphi_C \colon & A & \rightarrow & \enm_{\fq}(B) \\
      & \theta & \mapsto & \theta + \tau
   \end{array}, \]
where $\tau$ is the Frobenius endomorphism $b \mapsto b^q$.
For $P$ a prime of $A$ (\ie a monic irreducible polynomial), one can show (see \cite[theorem 3.6.3]{Gos}) that $[C(A/PA)]_A = P - 1$.
Thus 
\begin{equation} \label{eqn-zeta1}
 \zeta_A(1) = \prod\limits_{P\ \text{prime}} \left( 1 - \frac{1}{P} \right)^{-1} = \prod\limits_{P\ \text{prime}} \frac{[\lie(C)(A/PA)]_A}{[C(A/PA)]_A}. 
\end{equation}

Recently, Taelman \cite{Tae12} associates, to a Drinfeld module $\phi$ over the ring of integers $R$ of a finite extension of $K$, a finite $A$-module called the \emph{class module} $H(\phi/R)$ and an $L$-series value $L(\phi/R)$.
In particular, if $\phi$ is the Carlitz module and $R$ is $A$, thanks to \eqref{eqn-zeta1}, we have
\[ L(C/A) = \zeta_A(1). \]
These objects are related by a \emph{class formula}: $L(\phi/R)$ is equal to the product of $[H(\phi/R)]_A$ times a regulator (see theorem 1 of loc. cit.).

This class formula was generalized by Fang \cite{Fan14}, using the theory of shtukas and ideas of Vincent Lafforgue, to abelian $t$-modules over $A$, which are $n$-dimensional analogues of Drinfeld modules.
In particular, for $\cn$, the $n^{th}$ tensor power of the Carlitz module, introduced by Anderson and Thakur \cite{AndTha90}, we have
\[ L(\cn/A) = \zeta_A(n) \]
and this is related to a class module and a regulator as in the work of Taelman.

On  the other hand, Pellarin \cite{Pel12} introduced a new class of $L$-series.
Let $t_1,\dots,t_s$ be indeterminates over $\cinf$, the completion of a fixed algebraic closure of $\kinf$.
For each $1 \leq i \leq s$, let $\chi_{t_i} : A \rightarrow \fq[t_1,\dots,t_s]$ be the $\fq$-linear ring homomorphism defined by $\chi_{t_i}(\theta) = t_i$.
Then, the \emph{Pellarin's $L$-value} at a positive integer $n$ is defined as
\[ L(\chi_{t_1} \cdots \chi_{t_s},n) := \sum\limits_{a \in A_+} \frac{\chi_{t_1}(a) \cdots \chi_{t_s}(a)}{a^n} \in \fq[t_1,\dots,t_s] \otimes_{\fq} \kinf. \]

In this paper, inspired by ideas developed by Taelman in \cite{Tae12}, we prove a class formula for abelian $t$-modules over $\fq(t_1,\dots,t_s)[\theta]$.
In particular, for $s = 0$, we recover theorem 1.10 of \cite{Fan14}.
Then, we express Pellarin's $L$-values as a product of quotients of Fitting ideals in the manner of \eqref{eqn-zeta1}.
Thus, we obtain a class formula for these $L$-values (see section \ref{sec-cfPellarinL}).
This result was already used by Angl\`es, Pellarin and Tavares Ribeiro \cite{APTR14} in the $1$-dimensional case, \ie for Drinfeld modules.

Finally, let $a \in A_+$ be squarefree and $L$ be the cyclotomic field associated with $a$, \ie the finite extension of $K$ generated by the $a$-torsion of the Carlitz module.
It is a Galois extension of group $\Delta_a \simeq (A/aA)^\times$.
Let $\chi\colon (A/aA)^\times \rightarrow F^*$ be a homomorphism where $F$ is a finite extension of $\fq$.
The special value at a positive integer $n$ of Goss $L$-series associated to $\chi$ is defined as 
\[ L(n,\chi) := \sum\limits_{b \in A_+} \frac{\chi(\overline{b})}{b^n} \in F \otimes_{\fq} \kinf, \]
where $\overline{b}$ is the image of $b$ in $(A/aA)^\times$.
We can group all the $L(n,\chi)$ together in one equivariant $L$-value $L(n,\Delta_a)$.
Then, we prove an equivariant class formula for these $L$-values (see theorem \ref{thm-fcequi}), generalizing that of Angl\`es and Taelman \cite{AngTae12} in the case $n = 1$.

\subsection*{Acknowledgements}
The author sincerely thanks Bruno Angl\`es, Lenny Taelman and Floric Tavares Ribeiro for fruitful discussions and useful remarks.

\section{Anderson modules and class formula} \label{sec-andmodclassformula}

\subsection{Lattices} \label{sec-lattices}

Let $\fq$ be the finite field with $q$ elements and $\theta$ an indeterminate over $\fq$.
We denote by $A$ the polynomial ring $\fq[\theta]$ and by $K$ the fraction field of $A$.
Let $\infty$ be the unique place of $K$ which is a pole of $\theta$ and $v_\infty$ the discrete valuation of $K$ corresponding to this place with the normalization $v_\infty(\theta) = -1$.
The completion of $K$ at $\infty$ is denoted by $\kinf$.
We have $\kinf = \fq((\theta^{-1}))$.
We denote by $\cinf$ a fixed completion of an algebraic closure of $\kinf$.
The valuation on $\cinf$ that extends $v_\infty$ is still denoted by $v_\infty$.

Let $s \geq 0$ be an integer and $t_1,\dots,t_s$ indeterminates over $\cinf$.
We set $k_s := \fq(t_1,\dots,t_s)$, $R_s := k_s[\theta]$, $K_s := k_s(\theta)$ and $\ksinf := k_s((\theta^{-1}))$.
For $f \in \cinf[t_1,\dots,t_s]$ a polynomial expanded as a finite sum 
\[ f = \sum\limits_{i_1,\dots,i_s \in \mathbb N} \alpha_{i_1,\dots,i_s} t_1^{i_1} \cdots t_s^{i_s}, \]
with $\alpha_{i_1,\dots,i_s} \in \cinf$,
we set 
\[ v_\infty(f) := \inf\left\{ v_\infty(\alpha_{i_1,\dots,i_s}) \mid i_1,\dots,i_s \in \mathbb N \right\}. \]
For $f \in \cinf(t_1,\dots,t_s)$, there exists $g$ and $h$ in $\cinf[t_1,\dots,t_s]$ such that $f = g/h$, then we define $v_\infty(f) := v_\infty(g) - v_\infty(h)$.
We easily check that $v_\infty$ is a valuation, trivial on $k_s$, called the \emph{Gauss valuation}.
For $f \in \cinf[t_1,\dots,t_s]$, we set $\| f \|_\infty := q^{-v_\infty(f)}$ if $f \neq 0$ and $\| 0 \|_\infty = 0$.
The function $\|\cdot\|_\infty$ is called the \emph{Gauss norm}.
We denote by $\csinf$ the completion of $\cinf(t_1,\dots,t_s)$ with respect to $v_\infty$.

Let $V$ be a finite dimensional $\ksinf$-vector space and $\|\cdot\|$ be a norm on $V$ compatible with $\|\cdot\|_\infty$ on $\ksinf$.
For $r>0$, we denote by $B(0,r) := \left\{ v \in V \mid \|v\| < r \right\}$ the open ball of radius $r$, which is a $k_s$-subspace of $V$.


\begin{defi} \label{defi-lattices}
 A sub-$R_s$-module $M$ of $V$ is an $R_s$-\emph{lattice} of $V$ if it is free of rank one and the $\ksinf$-vector space spanned by $M$ is $V$.
\end{defi}

We can characterize these lattices.

\begin{lem} \label{lem-caralattices}
 Let $V$ be a $\ksinf$-vector space of dimension $n\geq 1$ and $M$ be a sub-$R_s$-module of~$V$.
 The following assertions are equivalent:
 \begin{enumerate}
  \item $M$ is an $R_s$-lattice of $V$;
  \item $M$ is discrete in $V$ and every open subspace of the $k_s$-vector space $V/M$ is of finite co-dimension.
 \end{enumerate}
\end{lem}
\begin{proof}
 Let us suppose that $M$ is an $R_s$-lattice of $V$, \ie there exists a family $(e_1,\dots,e_n)$ of elements of $M$ such that
 \[ M = \bigoplus\limits_{i = 1}^n R_s e_i \quad \text{and} \quad V = \bigoplus\limits_{i = 1}^n \ksinf e_i. \]
 Any element $v$ of $V$ can be uniquely written as $v = \sum\limits_{i = 1}^n v_i e_i$ with $v_i \in \ksinf$.
 Then, we set $\|v\| := \max\left\{ \|v_i\|_\infty \mid i = 1,\dots,n \right\}$.
 Since $R_s$ is discrete in $\ksinf$, this implies that $M$ is discrete in $V$.
 Now, let $m \geq 0$ be an integer.
 We have 
 \[ B \left(0,q^{-m}\right) = \bigoplus\limits_{i = 1}^n \theta^{- m - 1} k_s[[\theta^{-1}]] e_i. \]
 In particular, we have $V = M \oplus B(0,1)$ and
 \[ \dim_{k_s} \frac{B(0,q^{-m})}{B(0,q^{- m - 1})} = n. \]
 This implies that every open $k_s$-subspace of $V/M$ is of finite co-dimension.
 
 Reciprocally, let us suppose that $M$ is discrete in $V$ and every open subspace of the $k_s$-vector space $V/M$ is of finite co-dimension.
 Let $W$ be the $\ksinf$-subspace of $V$ generated by $M$ and $m$ be its dimension.
 There exist $e_1,\dots,e_m$ in $M$ such that
 \[ W = \bigoplus\limits_{i = 1}^m \ksinf e_i. \]
 Set 
 \[ N = \bigoplus\limits_{i = 1}^m R_s e_i. \]
 This is a sub-$R_s$-module of $M$ and an $R_s$-lattice of $W$.
 In particular, $M/N$ is discrete in $W/N$.
 Since any open $k_s$-subspace of $W/N$ is of finite co-dimension, we deduce that $M/N$ is a finite dimensional $k_s$-vector space.
 This implies that $M$ is a free $R_s$-module of rank $m$.
 Finally, observe that, if $m<n$, $V/M$ can not verify the co-dimensional property, thus $W=V$.
\end{proof}

\subsection{Anderson modules and exponential map} \label{sec-andmod}

Let $L$ be a finite extension of $K$, $L \subseteq \cinf$. 
We define $\rls$ to be the subring of $L_s := L(t_1,\dots,t_s)$ generated by $k_s$ and $\ol$, where $\ol$ is the integral closure of $A$ in $L$.
We set $\lsinf := L \otimes_K \ksinf$. This is a finite dimensional $\ksinf$-vector space.
We denote by $S_\infty(L)$ the set of places of $L$ above~$\infty$.
For a place $\nu \in S_\infty(L)$, we denote by $L_\nu$ the completion of $L$ with respect to $\nu$.
Let $\pi_\nu$ be a uniformizer of $L_\nu$ and $\mathbb F_\nu$ be the residue field of $L_\nu$. 
Then, we define $L_{s,\nu} := \mathbb F_\nu(t_1,\dots,t_s)((\pi_\nu))$ viewed as a subfield of $\csinf$.
We have an isomorphism of $\ksinf$-algebras
\[ \lsinf \simeq \prod\limits_{\nu \in S_\infty(L)} L_{s,\nu}. \]
Observe that $\rls$ is an $R_s$-lattice in the $\ksinf$-vector space $\lsinf$.

Let $\tau \colon \csinf \rightarrow \csinf$ be the morphism of $k_s$-algebras given by the $q$-power map on $\cinf$. 

\begin{lem} \label{lem-csinftau1}
 The elements of $\csinf$ fixed by $\tau$ are those of $k_s$.
\end{lem}
\begin{proof}
 Obviously, $k_s \subseteq \csinf^{\tau = 1}$.
 Reciprocally, observe that $\csinf^{\tau = 1} \subseteq \{ f \in \csinf \mid v_\infty(f) = 0 \}$.
 But we have the direct sum of $\fq[\tau]$-modules
 \[ \left\{ f \in \csinf \mid v_\infty(f) \geq 0 \right\} = \overline{\fq}(t_1,\dots,t_s) \oplus \left\{ f \in \csinf \mid v_\infty(f) > 0 \right\}. \]
 Since $\overline{\fq}(t_1,\dots,t_s)^{\tau = 1} = k_s$, we get the result. 
\end{proof}

The action of $\tau$ on $\lsinf = L \otimes_K \ksinf$ is the diagonal one $\tau \otimes \tau$.

\begin{defi} \label{defi-andersonmodule}
 Let $r$ be a positive integer. An \emph{Anderson module} $E$ over $\rls$ is a morphism of $k_s$-algebras
 \[ \begin{array}{cccc} 
    \phi_E \colon & R_s & \longrightarrow & M_n(\rls)\{\tau\} \\
    & \theta & \longmapsto & \displaystyle\sum\limits_{j = 0}^r A_j \tau^j \\
   \end{array} \]
 for some $A_0,\dots,A_r \in M_n(\rls)$ such that $(A_0 - \theta I_n)^n = 0$. 
\end{defi}

These objects are usually called \emph{abelian $t$-motives} as in the terminology of \cite{And86} but, to avoid confusion between $t$ and the indeterminates $t_1,\dots,t_s$, we prefer called them Anderson modules.
Note also that Drinfeld modules are one-dimensional Anderson modules.

For a matrix $A = (a_{ij}) \in M_n(\csinf)$, we set $v_\infty(A) := \min\limits_{1 \leq i,j \leq n}\left\{ v_\infty(a_{ij}) \right\}$ and $\tau(A) := (\tau(a_{ij})) \in M_n(\csinf)$.

\begin{prop} \label{prop-exp}
 There exists a unique skew power series $\expe := \sum\limits_{j \geq 0} e_j\tau^j$ with coefficients in $M_n(L_s)$ such that
 \begin{enumerate}
  \item $e_0 = I_n$;
  \item $\expe A_0 = \phi_E(\theta)\expe$ in $M_n(L_s)\{\{\tau\}\}$;
  \item $\lim\limits_{j\rightarrow\infty} \frac{v_\infty(e_j)}{q^j} = +\infty$.
 \end{enumerate}
\end{prop}
\begin{proof}
 See proposition 2.1.4 of \cite{And86}.
\end{proof}

Observe that $\expe$ is locally isometric. 
Indeed, by the third point, 
\[ c := \sup_{j \geq 1} \left( \frac{-v_\infty(e_j)}{q^j-1} \right) \]
is finite. 
Then, for any $x \in \lsinf^n$ such that $v_\infty(x)>c$, we have
\[ v_\infty\left(\sum\limits_{j\geq 0}e_j\tau^j(x)-x\right) \geq \min\limits_{j\geq 1}\left(v_\infty(e_j)+q^jv_\infty(x)\right) > v_\infty(x). \]

If $B$ is an $\rls$-algebra, we denote by $E(B)$ the $k_s$-vector space $B^n$ equipped with the structure of $R_s$-module induced by $\phi_E$.
We can also consider the tangent space $\lie(E)(B)$ which is the $k_s$-vector space $B^n$ whose $R_s$-module structure is given by the morphism of $k_s$-algebras
\[ \begin{array}{cccc} 
    \partial \colon & R_s & \longrightarrow & M_n(\rls) \\ 
     & \theta & \longmapsto & A_0
   \end{array}. \]
In particular, by the previous proposition, we get a continuous $R_s$-linear map
\[ \expe \colon \lie(E)(\lsinf) \longrightarrow E(\lsinf). \]

\subsection{The class formula} \label{sec-cnf}

In this section, we define a class module and two lattices in order to state the main result.

\begin{lem} \label{lem-fang} \
  \begin{enumerate}
  \item $A_0^{q^n} = \theta^{q^n}I_n$ ;
  \item $\inf\limits_{j \in \mathbb Z}\left( v_\infty(A_0^j) + j \right)$ is finite.
 \end{enumerate}
\end{lem}
\begin{proof}
 See lemma 1.4 of \cite{Fan14}. 
\end{proof}

By the second point, for any $a_j \in k_s$ and $m\in\mathbb Z$, the series $\sum\limits_{j\geq m} a_j A_0^{-j}$ converges in $M_n(\lsinf)$. 
Thus, $\partial$ can be uniquely extended to a morphism of $k_s$-algebras by
\[ \begin{array}{cccc} 
    \partial\colon & \ksinf & \longrightarrow & M_n(\lsinf) \\ 
    & \displaystyle \sum\limits_{j\geq m} a_j \frac{1}{\theta^j} & \longmapsto & \displaystyle \sum\limits_{j\geq m} a_j A_0^{-j}
   \end{array}, \]
where $a_j \in k_s$ and $m \in \mathbb Z$. 
Then, $\lie(E)(\lsinf)$ inherits a $\ksinf$-vector space structure.
Observe, by the first point of the lemma, that, for any $f\in k_s((\theta^{-q^n}))$, we have $\partial(f)=f I_n$, \ie the action is the scalar multiplication for these elements.
In particular, we get an isomorphism $\lie(E)(\lsinf) \simeq \lsinf^n$ as $k_s((\theta^{-q^n}))$-modules.
We deduce that $\lie(E)(\lsinf)$ is a $k_s((\theta^{-q^n}))$-vector space of dimension $nq^n$, so of dimension $n$ over $\ksinf$.


\begin{prop} \label{prop-lielattice}
 The $R_s$-module $\lie(E)(\rls)$ is an $R_s$-lattice of $\lie(E)(\lsinf)$.
 Furthermore, if $L = K$, the canonical basis is an $R_s$-base of $\lie(E)(R_s)$.
\end{prop}
\begin{proof}
 By the first point of the previous lemma, $\lie(E)(\rls)$ and $\rls^n$ are isomorphic as $k_s[\theta^{q^n}]$-modules.
 Thus, $\lie(E)(\rls)$ is a finitely generated $k_s[\theta^{q^n}]$-module.
 On the other hand, the action of an element $a \in R_s$ is the left multiplication by $aI_n + N$ where $N$ is a nilpotent matrix.
 Since $aI_n + N$ is an invertible matrix, $\lie(E)(\rls)$ is a torsion-free $R_s$-module.
 Moreover, the $k_s((\theta^{-q^n}))$-vector space generated by $\lie(E)(\rls)$ and $\ksinf$ is $\lsinf^n \simeq \lie(E)(\lsinf)$. 
 Therefore, $\lie(E)(\rls)$ is a free $R_s$-module of finite rank.
 Looking at the dimension as $K_s$-vector space, the rank is necessarily $n$.
 
 For the second assertion, denote by $e_i$ the $i^{th}$ vector of the canonical basis. 
 Firstly, we show that this family spans $\lie(E)(R_s)$.
 We proceed by induction on $\max\limits_{1 \leq i \leq n} \deg_\theta x_i$ where $(x_1, \dots, x_n)$ is in $R_s^n$.
 The case of degree $0$ is trivial because the action of an element of $k_s$ is the scalar multiplication.
 Now let $m$ be a positive integer and $(x_1, \dots, x_n)$ be a vector of $R_s^n$ such that $\max\limits_{1 \leq i \leq n} \deg_\theta x_i = m$.
 We can write
 \[ \begin{pmatrix} x_1 \\ \vdots \\ x_n \end{pmatrix} = \theta^m \begin{pmatrix} \zeta_1 \\ \vdots \\ \zeta_n \end{pmatrix}, \]
 where $\zeta_1, \dots, \zeta_n$ are elements of $k_s$.
 Since we have $\partial_{\theta^m} = \theta^m I_n \mod \theta^{m - 1}$, we get
 \[ \begin{pmatrix} x_1 \\ \vdots \\ x_n \end{pmatrix} = \partial_{\zeta_1 \theta^m} e_1 + \cdots + \partial_{\zeta_n \theta^m} e_n \mod \theta^{m - 1}. \]
 Thus we obtain the spanning property by induction.
 
 Finally, suppose that there exists $(a_1, \dots, a_n) \neq (0, \dots, 0)$ in $R_s^n$ such that 
 \[ \sum\limits_{i = 1}^n \partial_{a_i} e_i = 0. \]
 Let $d := \max\limits_{1 \leq i \leq n} \deg_\theta a_i$.
 Looking at the above equality modulo $\theta^d$, since $\partial_{\theta^d} = \theta^d I_n \mod \theta^{d - 1}$, we obtain that $a_i = 0$ if $\deg_\theta a_i = d $, thus necessarily all the $a_i$ are zero, \ie $e_1, \dots, e_n$ are linearly independent in $\lie(E)(R_s)$.
\end{proof}

\begin{prop} \label{prop-classmod} \
 \begin{enumerate}
  \item Set
  \[ H(E/\rls) := \frac{E(\lsinf)}{\expe(\lie(E)(\lsinf)) + E(\rls)}. \] 
  This is a finite dimensional $k_s$-vector space, thus a finitely generated $R_s$-module and a torsion $R_s$-module, called the \emph{class module}.
  \item The $R_s$-module $\expei(E(\rls))$ is an $R_s$-lattice in $\lie(E)(\rls)$.
 \end{enumerate}
\end{prop}
\begin{proof}
 Let $V$ be an open neighbourhood of $0$ in $\lsinf^n$ on which $\expe$ acts as an isometry and such that $\expe(V) = V$.
 We have a natural surjection of $k_s$-vector spaces
 \[ \frac{\lsinf^n}{\rls^n+V} \twoheadrightarrow H(E/\rls). \]
 By proposition \ref{prop-lielattice}, the left hand side is a finite dimensional $k_s$-vector space, hence \emph{a fortiori} $H(E/\rls)$ too.
 
 Now, let us prove that $\expei(E(\rls))$ is an $R_s$-lattice in $\lie(E)(\rls)$.
 Since the kernel of $\expe$ and $\lie(E)(\rls)$ are discrete in $\lie(E)(\rls)$, so is $\expei(E(\rls))$.
 Let $V$ be an open neighbourhood of $0$ on which $\expe$ is isometric and such that $\expe(V) = V$. 
 The exponential map induces a short exact sequence of $k_s$-vector spaces
 \[ 0 \longrightarrow \frac{\lie(E)(\rls)}{\expei(E(\rls)) + V} \stackrel{\expe}{\longrightarrow} \frac{E(\lsinf)}{E(\rls)+V} \longrightarrow H(E/\rls) \longrightarrow 0. \]
 Since the two last $k_s$-vector spaces are of finite dimension, the first one is of finite dimension too; thus $\expei(E(\rls))$ satisfies the co-dimensional property. 
\end{proof}

An element $f \in \ksinf$ is \emph{monic} if
\[ f = \frac{1}{\theta^m} + \sum\limits_{i > m} x_i\frac{1}{\theta^i}, \]
where $m \in \mathbb Z$ and $x_i \in k_s$. 
For an $R_s$-module $M$ which is a finite dimensional $k_s$-vector space, we denote by $[M]_{R_s}$ the monic generator of the Fitting ideal of $M$.

Let $V$ be a finite dimensional $\ksinf$-vector space.
Let $M_1$ and $M_2$ be two $R_s$-lattices in $V$. 
There exists $\sigma \in \gl(V)$ such that $\sigma(M_1) = M_2$. 
Then, we define $[M_1 : M_2]_{R_s}$ to be the unique monic representative of $k_s^\times \det\sigma$.

The aim of the next section is to prove a class formula $\emph{\`a la Taelman}$ for Anderson modules:

\begin{thm} \label{thm-cnf}
 Let $E$ be an Anderson module over $\rls$.
 The infinite product
 \[ L(E/\rls) := \prod\limits_{\substack{\m \text{ maximal} \\ \text{ideal of } \ol}}\frac{[\lie(E)(\rls/\m\rls)]_{R_s}}{[E(\rls/\m\rls)]_{R_s}} \]
 converges in $\ksinf$. Furthermore, we have 
 \[ L(E/\rls) = [\lie(E)(\rls) : \expei(E(\rls))]_{R_s}[H(E/\rls)]_{R_s}. \]
\end{thm}

\section{Proof of the class formula} \label{sec-proof}

The proof is very close to ideas developed by Taelman in \cite{Tae12} so we will only recall some statements and point out differencies.

\subsection{Nuclear operators and determinants} \label{sec-nuclearopdet}

Let $k$ be a field and $V$ a $k$-vector space equipped with a non-archimedean norm $\| \cdot \|$. 
Let $\varphi$ be a continuous endomorphism of $V$. 
We say that $\varphi$ is \emph{locally contracting} if there exist an non empty open subspace $U \subseteq V$ and a real number $0<c<1$ such that $\|\varphi(u)\|\leq c\| u\|$ for all $u\in U$. 
Any such open subspace U which moreover satisfies $\varphi(U) \subseteq U$ is called a \emph{nucleus} for $\varphi$. 
Observe that any finite collection of locally contracting endomorphisms of $V$ has a common nucleus. 
Furthermore if $\varphi$ and $\phi$ are locally contracting, then so are the sum $\varphi + \psi$ and the composition $\varphi\psi$. 

For every positive integer $N$, we denote by $V[[Z]]/Z^N$ the $k[[Z]]/Z^N$-module $V \otimes_k k[[Z]]/Z^N$ and by $V[[Z]]$ the $k[[Z]]$-module $V[[Z]] := \varprojlim V[[Z]]/Z^N$ equipped with the limit topology.
Observe that any continuous $k[[Z]]$-linear endomorphism $\Phi \colon V[[Z]] \rightarrow V[[Z]]$ is of the form
\[ \Phi = \sum\limits_{n \geq 0} \varphi_n Z^n, \]
where the $\varphi_n$ are continuous endomorphisms of $V$. 
Similarly, any continuous $k[[Z]]/Z^n$-linear endomorphism of $V[[Z]]/Z^N$ is of the form
\[ \sum\limits_{n = 0}^{N - 1} \varphi_n Z^n. \]
We say that the continuous $k[[Z]]$-linear endomorphism $\Phi$ of $V[[Z]]$ (resp. of $V[[Z]]/Z^N$) is \emph{nuclear} if for all $n$ (resp. for all $n<N$), the
endomorphism $\varphi_n$ of $V$ is locally contracting.

From now on, we assume that for any open subspace $U$ of $V$, the $k$-vector space $V/U$ is of finite dimension.

Let $\Phi$ be a nuclear endomorphism of $V[[Z]]/Z^N$. Let $U_1$ and $U_2$ be common nuclei for the $\varphi_n$, $n<N$. 
Since Proposition 8 in \cite{Tae12} is valid in our context,
\[ \det\nolimits_{k[[Z]]/Z^N}\left( 1 + \Phi \mid V/U_i \otimes_k k[[Z]]/Z^N \right) \in k[[Z]]/Z^N \]
is independent of $i \in \{1,2\}$. 
We denote this determinant by
\[ \det\nolimits_{k[[Z]]/Z^N}( 1 + \Phi \mid V). \]
If $\Phi$ is a nuclear endomorphism of $V[[Z]]$, then we denote by $\det\nolimits_{k[[Z]]} (1 + \Phi \mid V)$ the unique power series that reduces to $\det\nolimits_{k[[Z]]/Z^N} (1 + \Phi \mid V)$ modulo $Z^N$ for every $N$.

Note that Proposition 9, Proposition 10, Theorem 2 and Corollary 1 of \cite{Tae12} are also valid in our context. 
We recall the statements for the convenience of the reader.

\begin{prop} \label{prop-multidet} \
\begin{enumerate}
 \item Let $\Phi$ be a nuclear endomorphism of $V[[Z]]$. 
 Let $W \subseteq V$ be a closed subspace such that $\Phi(W[[Z]]) \subseteq W[[Z]]$. 
 Then $\Phi$ is nuclear on $W[[Z]]$ and $(V/W)[[Z]]$, and
 \[ \det\nolimits_{k[[Z]]} (1 + \Phi \mid V) = \det\nolimits_{k[[Z]]} ( 1 + \Phi \mid W) \det\nolimits_{k[[Z]]} (1 + \Phi \mid V/W). \]
 \item Let $\Phi$ and $\Psi$ be nuclear endomorphisms of $V[[Z]]$. 
 Then $(1 + \Phi)(1 + \Psi) - 1$ is nuclear, and
 \[ \det\nolimits_{k[[Z]]} ((1 + \Phi)(1 + \Psi) \mid V) = \det\nolimits_{k[[Z]]} (1 + \Phi \mid V) \det\nolimits_{k[[Z]]} (1 + \Psi \mid V). \]
\end{enumerate}
\end{prop}

\begin{thm} \label{thm-commudetegaux} \
\begin{enumerate}
 \item Let $\varphi$ and $\psi$ be continuous $k$-linear endomorphisms of $V$ such that $\varphi$, $\varphi\psi$ and $\psi\varphi$ are locally contracting. 
 Then 
 \[ \det\nolimits_{k[[Z]]} (1 + \varphi \psi Z \mid V) = \det\nolimits_{k[[Z]]} (1 + \psi \varphi Z \mid V). \]
 \item Let $N\geq 1$ be an integer. 
 Let $\varphi$ and $\psi$ be continuous $k$-linear endomorphisms of $V$ such that all compositions $\varphi$, $\varphi \psi$, $\psi \varphi$, $\varphi^2$, etc. in $\varphi$ and $\psi$ containing at least one endomorphism $\varphi$ and at most $N - 1$ endomorphisms $\psi$ are locally contracting. 
 Let $\Delta = \sum\limits_{n = 1}^{N - 1} \gamma_n Z^n$ such that
 \[ 1 + \Delta = \frac{1 - (1 + \varphi)\psi Z}{1 - \psi(1 + \varphi)Z} \mod Z^N. \]
 Then $\Delta$ is a nuclear endomorphism of $V[[Z]]$ and
 \[ \det\nolimits_{k[[Z]]} (1 + \Delta \mid V) = 1 \mod Z^N. \]
\end{enumerate}
\end{thm}

\subsection{Taelman's trace formula} \label{sec-traceformula}

Let $L$ be a finite extension of $K$ and $E$ be the Anderson module given by
 \[ \begin{array}{cccc} 
    \phi \colon & R_s & \longrightarrow & M_n(\rls)\{\tau\} \\
    & \theta & \longmapsto & \displaystyle\sum\limits_{j=0}^r A_j \tau^j \\
   \end{array} \]
for some $A_0,\dots,A_r \in M_n(\rls)$ such that $(A_0 - \theta I_n)^n = 0$. 
Let $M_n(\rls)\{\tau\}[[Z]]$ be the ring of formal power series in $Z$ with coefficients in $M_n(\rls)\{\tau\}$, the variable $Z$ being central.
 
We set 
\[ \Theta := \sum\limits_{n\geq 1} (\partial_\theta-\phi_\theta) \partial_\theta^{n-1} Z^n \in M_n(\rls)\{\tau\}[[Z]]. \]

\begin{lem} \label{lem-fittingasdet}
 Let $\m$ be a maximal ideal of $\ol$. 
 In $\ksinf$, the following equality holds: 
 \[ \frac{[\lie(E)(\rls/\m\rls)]_{R_s}}{[E(\rls/\m\rls)]_{R_s}} = \det\nolimits_{k_s[[Z]]} \left( 1 + \Theta \mid (\rls/\m\rls)^n \right)^{-1} \mid_{Z = \theta^{-1}}. \]
\end{lem}
\begin{proof}
 It is an easy computation using the definition of Fitting ideal and of $\Theta$.  
\end{proof}

Let $S$ be a finite set of places of $L$ containing $S_\infty(L)$.
Denote by $\os$ the ring of regular functions outside $S$.
In particular $\ol \subseteq \os$.
Let $\rss$ be the subring of $L_s$ generated by $\os$ and $k_s$.
For example, if $S = S_\infty(L)$, we have $\rss = \rls$.

Let $\p$ be a maximal ideal of $\ol$ which is not in $S$.
The natural inclusion $\ol \hookrightarrow \os$ induces an isomorphism $\rls/\p\rls \stackrel{\sim}{\longrightarrow} \rss/\p\rss$.
By the previous lemma, we obtain
\begin{equation} \label{eqn-fittingasdet} 
\frac{[\lie(E)(\rls/\p\rls)]_{R_s}}{[E(\rls/\p\rls)]_{R_s}} = \det\nolimits_{k_s[[Z]]} \left( 1 + \Theta \mid (\rss/\p\rss)^n \right)^{-1} \mid_{Z = \theta^{-1}}. 
\end{equation}

Denote by $\lsp$ the $\p$-adic completion of $L_s$, \ie the completion of $L_s$ with respect to the valuation $v_\p$ defined on $L[t_1,\dots,t_s]$ by
\[ v_\p \left(\sum\limits_{i_1,\dots,i_s \in \mathbb N} \alpha_{i_1,\dots,i_s} t_1^{i_1} \cdots t_s^{i_s} \right) 
 := \inf\limits_{i_1,\dots,i_s \in \mathbb N} \left\{ v_\p(\alpha_{i_1,\dots,i_s}) \right\}, \]
where $v_\p$ is the normalized $\p$-adic valuation on $L$.
Denote by $\osp$ the valuation ring of $\lsp$.
By the strong approximation theorem, for any $n > 0$, there exists $\pi_n \in L$ such that $v_\p(\pi_n) = -n$ and $v(\pi_n) \geq 0$ for all $v \notin S \cup \p$.
Thus, we have
\begin{equation} \label{eqn-lsprss}
 \lsp = \osp + \rsps \quad \text{and} \quad \rss = \osp + \rsps. 
\end{equation}

Finally, denote by $\lss$ the product of the completions of $L_s$ with respect to places of $S$.
For example, if $S = S_\infty(L)$, we have $\lss = \lsinf$.

%
%

Recall that $\rss$ is a Dedekind domain, discrete in $\lss$ and such that every open subspace of $\lss/\rss$ is of finite co-dimension.
Observe also that any element of $M_n(\rss)\{\tau\}$ induces a continuous $k_s$-linear endomorphism of $(\lss/\rss)^n$ which is locally contracting.
In particular, the endomorphism $\Theta$ is a nuclear operator of $(\lss/\rss)^n[[Z]]$.

\begin{lem} \label{lem-localization}
 Let $\p$ be a maximal ideal of $\ol$ which is not in $S$. Then
 \[ \det\nolimits_{k_s[[Z]]} \left( 1 + \Theta \mid (\rss/\p\rss)^n \right) = 
 \frac{ \det\nolimits_{k_s[[Z]]} \left( 1 + \Theta \mid \left( \frac{\lss\times\lsp}{\rsps} \right)^n \right) }{ \det\nolimits_{k_s[[Z]]} \left( 1 + \Theta \mid \left(\frac{\lss}{\rss} \right)^n \right) }. \]
\end{lem}
\begin{proof}
 The proof is the same as that of lemma 1 of \cite{Tae12}, using equalities \eqref{eqn-lsprss}. 
\end{proof}

\begin{prop} \label{prop-traceformula}
 The following equality holds in $\ksinf$:
 \[ L(E/\rls) = \det\nolimits_{k_s[[Z]]} \left( 1 + \Theta \mid (\lsinf/\rls)^n \right) \mid_{Z = \theta^{-1}}. \]
 In particular, $L(E/\rls)$ converges in $\ksinf$.
 
\end{prop}
\begin{proof}
 By lemma \ref{lem-fittingasdet}, we have
 \[ L(E/\rls) = \prod\limits_{\m} \det\nolimits_{k_s[[Z]]}\left(1+\Theta\mid (\rls/\m\rls)^n \right)^{-1}\mid_{Z=\theta^{-1}}, \]
 where the product runs through maximal ideals of $\ol$. 
 Fix $S \supseteq S_\infty(L)$ as above (the case $S = S_\infty(L)$ suffices).
 By equality \eqref{eqn-fittingasdet}, we have
 \[ \prod\limits_{\m} \det\nolimits_{k_s[[Z]]} \left( 1 + \Theta \mid (\rls/\m\rls)^n \right)^{-1}
  = \prod\limits_{\m} \det\nolimits_{k_s[[Z]]} \left( 1 + \Theta \mid (\rss/\m\rss)^n \right)^{-1}, \]
 where the products run through maximal ideals of $\ol$ which are not in $S$.
 
 Define $S_{D,N}$ as in \cite{Tae12}.
 It suffices to prove that for any $1 + F \in S_{D,N}$, the infinite product
 \[ \prod\limits_{\m \notin S \setminus S_\infty(L)} \det\nolimits_{k_s[[Z]]/Z^N} \left( 1 + F \mid \left( \frac{\rss}{\m\rss} \right)^n \right) \] 
 converges to
 \[ \det\nolimits_{k_s[[Z]]/Z^N} \left( 1 + F \mid \left(\frac{\lss}{\rss}\right)^n \right)^{-1}. \]
 
 Let $\m_1,\dots,\m_r$ be the maximal ideals of $\ol$ which are not in $S$ and such that $\m_i\rss$ is a maximal ideal of $\rss$ verifying $\dim_{k_s} \rss/\m_i\rss < D$.
 Applying successively lemma \ref{lem-localization} to $\rss$, $R_{S \cup \{\m_1\},s}$, $R_{S \cup \{\m_1,\m_2\},s}$, etc., we obtain the following equality: 
 \[ \begin{array}{c} 
 \displaystyle \det\nolimits_{k_s[[Z]]} \left( 1 + F \mid \left( \frac{\lss}{\rss} \right)^{\hspace{-0.1cm}n} \right) 
 \prod\limits_{\m} \det\nolimits_{k_s[[Z]]} \left( 1 + F \mid \left( \frac{\rss}{\m\rss} \right)^{\hspace{-0.1cm}n} \right) = \\
 \displaystyle \!\det\nolimits_{k_s[[Z]]} \left( 1 + F \mid \left( \frac{\lss\times L_{s,\m_1} \times \cdots \times L_{s,\m_r}}{R_{S \cup \{ \m_1,\dots,\m_r \},s}} \right)^{\hspace{-0.1cm}n} \right)\!
 \!\prod\limits_{\m \neq \m_1,\dots,\m_r} \hspace{-0.25cm} \det\nolimits_{k_s[[Z]]} \left( 1 + F \mid \left( \frac{\rss}{\m\rss} \right)^{\hspace{-0.1cm}n} \right)\!. 
 \end{array} \]
 This allows us, replacing $\rss$ by $R_{S \cup \{ \m_1,\dots,\m_r \},s}$, to suppose that $\rss$ has not maximal ideal of the form $\m\rss$ with $\m$ maximal ideal of $\ol$ which is not in $S$ such that $\dim_{k_s} \rss/\m\rss < D$. 
 Then, we can finish the proof as in \cite{Tae12}. 
\end{proof}

\subsection{Ratio of co-volumes} \label{sec-ratio}

Let $V$ be a finite dimensional $\ksinf$-vector space and $\| \cdot \|$ be a norm on $V$ compatible with $\| \cdot \|_\infty$ on $\ksinf$.
Let $M_1$ and $M_2$ be two $R_s$-lattices in $V$ and $N \in \mathbb N$.
A continuous $k_s$-linear map $\gamma\colon V/M_1 \rightarrow V/M_2$ is $N$-\emph{tangent to the identity} on $V$ if there exists an open $k_s$-subspace $U$ of $V$ such that 
\begin{enumerate}
 \item $U \cap M_1 = U \cap M_2 = \{0\}$;
 \item $\gamma$ restricts to an isometry between the images of $U$;
 \item for any $u\in U$, we have $\| \gamma(u) - u \| \leq q^{-N} \| u \|$.
\end{enumerate}
The map $\gamma$ is \emph{infinitely tangent to the identity} on $V$ if it is $N$-tangent for every positive integer~$N$.

\begin{prop} \label{prop-seriestangent}
 Let $\gamma \in M_n(L_s)\{\{\tau\}\}$ be a power series convergent on $\lsinf^n$ with constant term equal to $1$ and such that $\gamma(M_1) \subseteq M_2$.
 Then $\gamma$ is infinitely tangent to the identity on $\lsinf^n$.
\end{prop}
\begin{proof}
 See proposition 12 of \cite{Tae12}. 
\end{proof}

For example, by proposition \ref{prop-exp}, the map 
\[ \expe \colon \frac{\lie(E)(\lsinf)}{\expei(E(\rls))} \longrightarrow \frac{E(\lsinf)}{E(\rls)} \]
is infinitely tangent to the identity on $\lsinf^n$.

Now, let $H_1$ and $H_2$ two finite dimensional $k_s$-vector spaces which are also $R_s$-modules and set $N_i := \frac{V}{M_i} \times H_i$ for $i = 1, 2$.
A $k_s$-linear map $\gamma \colon N_1 \rightarrow N_2$ is $N$-\emph{tangent} (resp. \emph{infinitely tangent}) to the identity on $V$ if the composition
\[ \frac{V}{M_1} \hookrightarrow N_1 \stackrel{\gamma}{\longrightarrow} N_2 \twoheadrightarrow \frac{V}{M_2} \]
is so. 
For a $k_s$-linear isomorphism $\gamma \colon N_1 \rightarrow N_2$, we define an endomorphism 
\[ \Delta_\gamma := \frac{1 - \gamma^{-1} \partial_\theta \gamma Z}{1 - \partial_\theta Z} - 1 
= \sum\limits_{i \geq 1} (\partial_\theta - \gamma^{-1} \partial_\theta \gamma) \partial^{n - 1} Z^n \]
of $N_1[[Z]]$.

\begin{prop} \label{prop-detformula}
 If $\gamma$ is infinitely tangent to the identity on $V$, then $\Delta_\gamma$ is nuclear and 
 \[ \det\nolimits_{k_s[[Z]]} (1 + \Delta_\gamma \mid N_1) \mid_{Z = \theta^{-1}} = [M_1 : M_2]_{R_s} \frac{[H_2]_{R_s}}{[H_1]_{R_s}}. \]
\end{prop}
\begin{proof}
 See theorem 4 of \cite{Tae12}. 
\end{proof}

\subsection{Proof of theorem \ref{thm-cnf}} \label{sec-endproof}

By theorem \ref{prop-traceformula}, $L(E/\rls)$ converges in $\ksinf$ and 
\[ L(E/\rls) = \det\nolimits_{k_s[[Z]]} \left( 1 + \Theta \mid (\lsinf/\rls)^n \right) \mid_{Z = \theta^{-1}}. \]
The exponential map $\expe$ induces a short exact sequence of $R_s$-modules
\[ 0 \longrightarrow \frac{\lie(E)(\lsinf)}{\expei(E(\rls))} \longrightarrow \frac{E(\lsinf)}{E(\rls)} \longrightarrow H(E/\rls) \longrightarrow 0. \]
By proposition \ref{prop-classmod}, the $k_s$-vector space $H(E/\rls)$ is of finite dimension. 
Moreover, since the $R_s$-module on the left is divisible and $R_s$ is principal, the sequence splits.
The choice of a section gives rise to an isomorphism of $R_s$-modules
\[ \frac{\lie(E)(\lsinf)}{\expei(E(\rls))} \times H(E/\rls) \simeq \frac{E(\lsinf)}{E(\rls)}. \]
This isomorphism can be restricted to an isomorphism of $k_s$-vector space
\[ \gamma \colon \frac{\lie(E)(\lsinf)}{\expei(E(\rls))} \times H(E/\rls) \stackrel{\sim}{\longrightarrow} \left( \frac{\lsinf}{\rls} \right)^n. \]
Observe that $\gamma$ corresponds with the map induced by $\expe$.
By proposition \ref{prop-seriestangent}, $\gamma$ is infinitely tangent to the identity on $\lsinf^n$.
By second point of proposition \ref{prop-exp} , we have $\expe \partial_\theta \expei = \phi_\theta$, hence the equality of $k_s[[Z]]$-linear endomorphisms of $\left( \frac{\lsinf}{\rls} \right)^n\![[Z]]$:
\[ 1 + \Theta = \frac{1 - \gamma \partial_\theta \gamma^{-1} Z}{1 - \partial_\theta Z}. \]
Thus, by theorem \ref{prop-detformula}, we obtain
\[ \det\nolimits_{k_s[[Z]]} ( 1 + \Theta \mid (\lsinf/\rls)^n ) \mid_{Z = \theta^{-1}} = [\lie(E)(\rls) : \expei(E(\rls))]_{R_s}[H(E/\rls)]_{R_s}. \]
This concludes the proof.

\section{Applications} \label{sec-appli}

\subsection{The \texorpdfstring{$n^{th}$}{n-th} tensor power of the Carlitz module} \label{sec-nthcarlitz}

Let $\alpha$ be a non-zero element of $R_s$. 
Let $E_\alpha$ be the Anderson module defined by the morphism of $k_s$-algebras $\phi \colon R_s \rightarrow M_n(R_s)\{\tau\}$ given by 
\[ \phi_\theta = \partial_\theta + N_\alpha\tau, \]
where 
\[ \partial_\theta = \begin{pmatrix}
        \theta & 1 & \cdots & 0 \\
        0 & \ddots & \ddots & \vdots \\
        \vdots & \ddots & \ddots & 1 \\
        0 & \cdots & 0 & \theta \\
       \end{pmatrix}
\quad \text{and} \quad 
N_\alpha = \begin{pmatrix} 
	0 & \cdots & \cdots & 0 \\ 
        \vdots & & & \vdots \\
        0  & & & \vdots \\
        \alpha  & 0 & \cdots & 0 \\
       \end{pmatrix}.
\]
In other words, if ${}^t(x_1,\dots,x_n)\in\csinf^n$, we have
\[ \phi_\theta\begin{pmatrix} x_1 \\ \vdots \\ x_n \end{pmatrix} = \begin{pmatrix} \theta x_1+x_2 \\ \vdots \\ \theta x_{n-1} + x_n \\ \theta x_n + \alpha\tau(x_1) \end{pmatrix}. \]

The case $\alpha = 1$ is denoted by $\cn$, the $n^{th}$ tensor power of Carlitz module, introduced in \cite{AndTha90}. 
In this section, we show that the exponential map associated to $\cn$ is surjective on $\csinf^n$ and we recall its kernel.

\subsubsection{Surjectivity and kernel of \texorpdfstring{$\expcn$}{expcn}} \label{sec-expcn}

By proposition \ref{prop-exp}, there exists a unique exponential map $\expcn$ associated with $\cn$ and by \cite[section 2]{AndTha90}, there exists a unique formal power series
\[ \logcn = \sum\limits_{i \geq 0} P_i \tau^i \in M_n(\csinf)\{\{\tau\}\} \]
such that $P_0 = I_n$ and $\logcn \cnt = \partial_\theta \logcn$.
These two maps are inverses of each other, \ie we have the equality of formal power series
\[ \logcn \expcn = \expcn \logcn = I_n. \]
Furthermore, by \cite[proposition 2.4.2 and 2.4.3]{AndTha90}, the series $\expcn(f)$ converges for all $f \in \csinf^n$ and $\logcn(f)$ for all $f = (f_1,\dots,f_n) \in \csinf^n$ such that $v_\infty(f_i) > n - i - \frac{nq}{q-1}$ for $1 \leq i \leq n$.

For an $n$-tuple $(r_1,\dots,r_n)$ of real numbers, we denote by $D_n(r_i, i = 1,\dots,n)$ the polydisc
\[ \left\{ f \in \csinf^n \mid v_\infty(f_i) > r_i,\ i = 1,\dots, n \right\}. \]

\begin{prop} \label{prop-expsurj}
 The exponential map $\expcn$ is surjective on $\csinf^n$.
\end{prop}

To prove this, we reduce to the one dimensional case.

\begin{lem} \label{lem-equivsurj}
 The following assertions are equivalent:
 \begin{enumerate}
  \item $\expcn$ is surjective on $\csinf^n$;
  \item $\cnt$ is surjective on $\csinf^n$;
  \item $\tau - 1$ is surjective on $\csinf$.
 \end{enumerate}
\end{lem}
\begin{proof}
 It is easy to show that that (1) implies (2).
 Indeed, let $y \in \csinf^n$.
 By hypothesis, there exists $x \in \csinf^n$ such that $\expcn(x) = y$. 
 Hence we have
 \[ \cnt \expcn (\partial_\theta^{-1} x) = \expcn(x) = y. \]
 
 Next we prove that (2) implies (3).
 Since $\cnt$ is supposed to be surjective on $\csinf^n$, for any $y = (y_1,\dots,y_n) \in \csinf^n$, there exists $x = (x_1,\dots,x_n) \in \csinf^n$ such that 
 \[ \left\{ \begin{array}{ccc} 
            \theta x_1 + x_2 & = & y_1 \\ 
            & \vdots & \\
            \theta x_{n - 1} + x_n & = & y_{n - 1} \\
            \theta x_n + \tau(x_1) & = & y_n \\ 
    \end{array} \right.. \]
 In particular, we get
 \begin{equation} \label{eqn-tauxsumy} 
  \tau(x_1) - (-\theta)^n x_1 = \sum\limits_{i = 1}^n (-\theta)^{n - i} y_i.  
 \end{equation}
 Thus $\tau - (-\theta)^n$ is surjective on $\csinf$.
 But we have
 \[ \tau \left( (-\theta)^{\frac{n}{q-1}} \right) = (-\theta)^n (-\theta)^{\frac{n}{q - 1}}, \]
 hence $\tau - 1$ is also surjective on $\csinf$.

 In fact, it is also easy to check that (3) implies (2). 
 As in the previous case, the surjectivity of $\tau - (-\theta)^n$ is deduced from the surjectivity of $\tau - 1$.
 Hence, for a fixed $y = (y_1,\dots,y_n) \in \csinf^n$, there exists $x_1 \in \csinf$ verifying equation \eqref{eqn-tauxsumy}.
 Then, by back-substitution, we find successively $x_2,\dots,x_n \in \csinf$ such that $x = (x_1,\dots,x_n)$ satisfies $\cnt(x) = y$.

 We finally prove that (2) implies (1). 
 Since $\logcn$ converges on the polydisc $D_n(n - i - \frac{nq}{q - 1}, i = 1,\dots,n)$ and $\expcn \logcn$ is the identity map on it, this polydisc is included in the image of the exponential.
 We will "grow" this polydisc to show that $\expcn$ is surjective.
 For $i = 1,\dots,n$, we define
 \[ r_{0, i} := n - i - \frac{nq}{q-1} = - i - \frac{n}{q-1}, \]
 and for $k \geq 1$, 
 \[ r_{k + 1, i} = \left\{\begin{array}{cl} r_{k, i + 1} & \text{if}\ 1 \leq i \leq n - 1 \\ 
  q r_{k, 1} & \text{if}\ i = n \end{array} \right..\]
 By induction, we prove that for any integer $k \geq 0$ and any $1 \leq i \leq n - 1$,
 \begin{equation} \label{eqn-rki+1}
  r_{k,i + 1} \leq r_{k,i} - 1. 
 \end{equation}
 We also prove that for any integer $k \geq 0$ and $i \in \{1,\dots,n\}$, we have $r_{k,i} \leq r_{0,i} - k$.
 In particular, for any $1 \leq i \leq n$, the sequence $(r_{k,i})$ tends to $- \infty$, \ie the polydiscs $D_n(r_{k,i}, i = 1,\dots,n)$ cover $\csinf^n$.
 Thus, it suffices to show that $D_n(r_{k,i}, i = 1,\dots,n) \subseteq \im \expcn$ for any integer $k \geq 0$.
 
  The case $k = 0$, corresponding to the convergence domain of $\logcn$, is already known.
 Let us suppose that $D_n(r_{k,i}, i = 1,\dots,n)$ is included in the image of $\expcn$ for an integer $k \geq 0$.
 Let $y$ be an element of $D_n(r_{k + 1, i}, i = 1,\dots,n) \setminus D_n(r_{k,i}, i = 1,\dots,n)$.
 
 We claim that there exists $x \in D_n(r_{k, i}, i = 1,\dots,n)$ such that $\cnt(x) = y$.
 
 Assume temporally this.
 Since $D_n(r_{k, i}, i = 1,\dots,n) \subseteq \im \expcn$, there exists $z \in \csinf^n$ such that $\expcn(z) = x$. 
 Thus 
 \[ \expcn(\partial_\theta z) = \cnt \expcn(z) = \cnt(x) = y. \]
 In particular $y$ is in the image of the exponential as expected.
 
 It only remains to prove the claim. 
 By hypothesis, there exists $x = (x_1,\dots,x_n) \in \csinf^n$ such that
 \[ \left\{ \begin{array}{ccc}
     x_2 & = & y_1 - \theta x_1 \\
     & \vdots & \\
     x_n & = & y_{n - 1} - \theta x_n \\
     \tau(x_1) - (-\theta)^n x_1 & = & \sum\limits_{i = 1}^n (-\theta)^{n - i} y_i
    \end{array}\right.. \]
 We need to show that $x$ is in $D_n(r_{k, i}, i = 1,\dots,n)$.
 Let begin by showing $v_\infty(x_1) > r_{k, 1}$.
 If $v_\infty(x_1) = \frac{-n}{q - 1}$, then $v_\infty(x_1) > r_{0, 1} > r_{k, 1}$. 
 So we may suppose that $v_\infty(x_1) \neq \frac{-n}{q - 1}$.
 Then 
 \[ v_\infty(\tau(x_1) - (-\theta)^n x_1) = \min(q v_\infty(x_1)\ ;\ v_\infty(x_1) - n). \]
 In particular,
 \[ q v_\infty(x_1) \geq v_\infty\left( \sum\limits_{i = 1}^n (-\theta)^{n - i} y_i \right) \geq \underset{1 \leq i \leq n}{\min}(v_\infty(y_i) - n + i) > \underset{1 \leq i \leq n}{\min}(r_{k + 1, i} - n + i), \]
 where the last inequality comes from the fact that $y$ is in $D_n(r_{k + 1, i}, i = 1,\dots,n)$. 
 But, by the inequality \eqref{eqn-rki+1}, we have
 \[ r_{k + 1, n} \leq r_{k + 1, n - 1} - 1 \leq \cdots \leq r_{k + 1, 1} - n + 1. \]
 Hence we get
 \[ q v_\infty(x_1) > r_{k + 1, n} = q r_{k,1}, \]
 as desired.
 
 Finally, we show that $v_\infty(x_i) > r_{k, i}$ for all $2 \leq i \leq n$.
 Since $y \in D_n(r_{k + 1, i}, i = 1,\dots,n)$, we have 
 \[ v_\infty(x_2) \geq \min(v_\infty(y_1)\ ;\ v_\infty(x_1) - 1) > \min (r_{k + 1, 1}\ ;\ r_{k, 1} - 1) = r_{k,2}, \]
 where the last equality comes from the definition of $r_{k + 1}$ and from inequality \eqref{eqn-rki+1}.
 On the same way, we obtain the others needed inequalities.  
\end{proof}

\begin{lem} \label{lem-tau-1surj}
 The application $\tau - 1 \colon \csinf \rightarrow \csinf$ is surjective.
\end{lem}
\begin{proof}
 Since $\sum\limits_{i \geq 0} \tau^i(x)$ converges for $x \in \csinf$ such that $v_\infty(x) > 0$, we have
 \[ \{ x \in \csinf \mid v_\infty(x) > 0 \} \subseteq \im(\tau - 1). \]
 Thus, since $\cinf(t_1,\dots,t_s)$ is dense in $\csinf$, it suffices to show that $\cinf(t_1,\dots,t_s) \subseteq (\tau - 1)(\csinf)$.
 Observe that $(\tau - 1)(\cinf[t_1,\dots,t_s]) = \cinf[t_1,\dots,t_s]$. 
 Now let $f \in \cinf(t_1,\dots,t_s)$.
 We can write
 \[ f = \frac{g}{h} \quad \text{with}\ g, h \in \cinf[t_1,\dots,t_s] \ \text{and}\ v_\infty(h) = 0. \]
 Now write $h = \delta - z$ with $\delta \in \overline{\fq}[t_1,\dots,t_s] \setminus \{0\}$ and $z \in \cinf[t_1,\dots,t_s]$ such that $v_\infty(z) > 0$.
 Then, in $\csinf$, we have
 \[ f = \frac{g}{h} = \sum\limits_{k \geq 0} \frac{g z^k}{\delta^{k + 1}}. \]
 On the one hand, since the series converges, there exists $k_0 \in \mathbb N$ such that
 \[ v_\infty \left( \sum\limits_{k \geq k_0} \frac{g z^k}{\delta^{k + 1}} \right) > 0. \]
 In particular, this sum is in the image of $\tau - 1$.
 On the other hand, we have
 \[ \sum\limits_{k = 0}^{k_0 - 1} \frac{g z^k}{\delta^{k + 1}} \in \frac{1}{\delta^{k_0}} \cinf[t_1,\dots,t_s]. \]
 But we can write $\frac{1}{\delta^{k_0}} = \frac{\beta}{\gamma}$ with $\beta \in \overline{\fq}[t_1,\dots,t_s]$ and $\gamma \in \fq[t_1,\dots,t_s] \setminus \{0\}$. 
 Hence
 \[ \sum\limits_{k = 0}^{k_0 - 1}\frac{g z^k}{\delta^{k + 1}} \in \frac{1}{\gamma} \cinf[t_1,\dots,t_s] \subseteq (\tau - 1)\left( \frac{1}{\gamma} \cinf[t_1,\dots,t_s] \right). \]
 Thus, by linearity of $\tau - 1$, we get $f \in (\tau - 1)(\csinf)$. 
\end{proof}

Denote by $\Lambda_n$ the kernel of the morphism of $R_s$-modules
\[ \expcn\colon \lie(\cn)(\csinf) \longrightarrow \cn(\csinf). \]

\begin{prop} \label{prop-kerexpcn}
 The $R_s$-module $\Lambda_n$ is free of rank $1$ and is generated by a vector with $\tilde\pi^n$ as last coordinate.
\end{prop}
\begin{proof}
 See \cite[section 2.5]{AndTha90}. 
\end{proof}

\subsubsection{Characterization of Anderson modules isomorphic to \texorpdfstring{$\cn$}{Cn}} \label{sec-isoandmod}

We characterize Anderson modules which are isomorphic, in a sense described below, to the $n^{th}$ tensor power of the Carlitz module. 
We obtain an $n$-dimensional analogue of proposition 6.2 of \cite{APTR14}.

\begin{defi} \label{defi-andmodisomorphic}
 Two Anderson modules $E$ and $E'$ are \emph{isomorphic} if there exists a matrix $P \in \gln(\csinf)$ such that $E_\theta P = P E_\theta'$ in $M_n(\csinf)\{\tau\}$.
\end{defi}

Let $\alpha \in R_s$.
Denote by $E_\alpha$ the Anderson module defined at the beginning of section \ref{sec-nthcarlitz}. 
Note that $E_\alpha$ and $\cn$ are isomorphic if and only if there exists a matrix $P \in \gln(\csinf)$ such that
\begin{equation} \label{eqn-charaiso} 
 \partial_\theta P = P \partial_\theta \quad \text{and} \quad N_1 \tau(P) = P N_\alpha. 
\end{equation}

Let us set 
\[ \calus := \left\{ \alpha \in \csinf^* \mid \exists \beta \in \cinf^*, \gamma \in \overline{\fq}(t_1,\dots,t_s), v_\infty \left(\alpha - \beta \frac{\tau(\gamma)}{\gamma} \right) > v_\infty(\alpha) \right\}. \]

\begin{lem} \label{lem-sesUs}
 The map which associates to any element $x$ of $\csinf^*$ the element $\frac{\tau(x)}{x}$ of $\csinf^*$ induces a short exact sequence of multiplicative groups
 \[ 1 \longrightarrow k_s^* \longrightarrow \csinf^* \longrightarrow \calus \longrightarrow 1. \]
\end{lem}
\begin{proof}
 The kernel comes from lemma \ref{lem-csinftau1}.

 Let $\alpha \in \csinf^*$ such that there exists $x \in \csinf^*$ verifying $\tau(x) = \alpha x$.
 Since $\cinf$ is an algebraically closed field, one can suppose that $v_\infty(\alpha) = 0$.
 We write $x = \gamma + m$ with $\gamma \in \overline{\fq}(t_1,\dots,t_s)$ and $m \in \csinf^*$ such that $v_\infty(m) > 0$.
 Then, we have $v_\infty(\tau(\gamma) - \alpha \gamma) > 0$, \ie $\alpha \in \calus$.
 
 Reciprocally, let $\alpha \in \calus$ and $\beta \in \cinf^*$, $\gamma \in \overline{\fq}(t_1,\dots,t_s)$ such that
 \[ v_\infty \left(\alpha - \beta \frac{\tau(\gamma)}{\gamma} \right) > v_\infty(\alpha) . \]
 We set $\delta := \beta \frac{\tau(\gamma)}{\gamma}$.
 Observe that $\prod\limits_{i \geq 0} \frac{\tau^i(\delta)}{\tau^i(\alpha)}$ converges in $\csinf^*$.
 Now, since $\tau$ is $k_s$-linear, there exists $\varepsilon \in \cinf^* \overline{\fq}(t_1,\dots,t_s)$ such that $\tau(\varepsilon) = \delta$.
 Then, we set
 \begin{equation} \label{eqn-omegaalpha}
  \omega_\alpha := \varepsilon \prod\limits_{i \geq 0} \frac{\tau^i(\delta)}{\tau^i(\alpha)} \in \csinf^*.
 \end{equation}
 Thus, we have $\tau(\omega_\alpha) = \alpha \omega_\alpha$.
 Observe that $\omega_\alpha$ is defined up to a scalar factor in $\fq^*$ whereas it depends \emph{a priori} on the choices of $\beta$, $\gamma$ and $\varepsilon$.  
\end{proof}

We are now able to characterize Anderson modules which are isomorphic to $\cn$.

\begin{prop} \label{prop-andmodiso}
 The following assertion are equivalent:
 \begin{enumerate}
  \item $E_\alpha$ is isomorphic to $\cn$, 
  \item $\alpha \in \calus$,
  \item $\expa$ is surjective,
  \item $\ker \expa$ is a free $R_s$-module of rank $1$,
 \end{enumerate}
 where $\expa$ is the exponential map associated with $E_\alpha$ by proposition \ref{prop-exp}.
\end{prop}
\begin{proof}
 Setting $P = \omega_\alpha I_n$ where $\omega_\alpha$ is defined by \eqref{eqn-omegaalpha} , we see that (2) implies (1).
 
 We prove that (1) implies (3). 
 Let $P \in \gln(\csinf)$ such that $\cnt P = P E_\theta$.
 Using equalities \eqref{eqn-charaiso}, we check that
 \[ P^{-1} \expcn P \partial_\theta = E_\theta P^{-1} \expcn P. \]
 Thus, by unicity in proposition \ref{prop-exp}, we get $P^{-1} \expcn P = \expa$.
 In particular, by proposition~\ref{prop-expsurj}, we deduce that $\expa$ is surjective.
 
 Next, we prove that (3) implies (2).
 We can assume that $v_\infty(\alpha) = 0$.
 By lemma \ref{lem-sesUs}, it suffices to show that $\ker(\alpha \tau - 1)$ is not trivial.
 Let us suppose the converse.
 As at the beginning of the proof of lemma \ref{lem-equivsurj}, we easily show that the surjectivity of $\expa$ on $\csinf^n$ implies that of $\alpha \tau - 1$ on $\csinf$.
 Thus, $\alpha \tau - 1$ is an automorphism of the $k_s$-vector space $\csinf$.
 We verify that $v_\infty(f) = 0$ if and only if $v_\infty(\alpha \tau(f) - f) = 0$.
 Let $\overline\alpha \in \overline{\fq}(t_1,\dots,t_s)$ such that $v_\infty(\alpha - \overline\alpha) > 0$. 
 Then, $\overline\alpha \tau - 1$ is an automorphism of the $k_s$-vector space $\overline{\fq}(t_1,\dots,t_s)$, which is obviously false.

 It is easy to show that (1) implies (4).
 Indeed, since $E_\alpha$ is isomorphic to $\cn$, we have
 \[ \ker\expa = \frac{1}{\omega_\alpha} \ker \expcn. \]
 Thus, by proposition \ref{prop-kerexpcn}, $\ker \expa$ is a free $R_s$-module of rank $1$ generated by a vector with $\frac{\tilde\pi^n}{\omega_\alpha}$ as last coordinate.
 
 Finally, we prove that (4) implies (2).
 Let $f$ be a non zero element of $\ker \expa$ such that $\partial_\theta^{-1}f \notin \ker \expa$.
 Thus, the vector $g := \expa(\partial_\theta^{-1} f) \in \csinf^n$ is non zero and $E_\theta(g) = 0$.
 Denote by $g_1,\dots,g_n$ its coordinates. 
 We have
 \[ \left\{ \begin{array}{ccc} 
            \theta g_1 + g_2 & = & 0 \\ 
            & \vdots & \\
            \theta g_{n - 1} + g_n & = & 0 \\
            \theta g_n + \alpha \tau(g_1) & = & 0 \\ 
           \end{array} \right.. \]
 As $g \neq 0$, we deduce that $g_i \neq 0$ for all $1 \leq i \leq n$.
 Summing, we obtain $\alpha \tau(g_1) - (-\theta)^n g_1 = 0$.
 Thus
 \[ \alpha \tau \left( (-\theta)^{\frac{-n}{q - 1}} g_1 \right) = (-\theta)^{\frac{-n}{q - 1}} g_1 . \]
 We conclude, by lemma \ref{lem-sesUs}, that $\alpha \in \calus$. 
\end{proof}


\begin{exe}
 Looking at the degree in $t_1$, we easily show that $t_1 \notin \calus$.
 So $E_{t_1}$ is not isomorphic to $\cn$ and $\exp_{t_1}$ is not surjective.
\end{exe}

\subsection{Pellarin's \texorpdfstring{$L$}{L}-functions} \label{sec-PellarinL}

Let $\alpha \in R_s \setminus \{0\}$ and $E_\alpha$ be the Anderson module defined at the beginning of section \ref{sec-nthcarlitz}.
By theorem \ref{thm-cnf}, we have a class formula for
\[ L(E_\alpha/R_s) := \prod\limits_{\substack{P \in A \\ \text{prime}}} \frac{[\lie(E_\alpha)(R_s/PR_s)]_{R_s}}{[E_\alpha(R_s/PR_s)]_{R_s}}. \]
We compute the $R_s$-module structure of $\lie(E_\alpha)(R_s/PR_s)$ and $E_\alpha(R_s/PR_s)$.
Then, we show that we recover special values of Pellarin's $L$-functions if we take $\alpha = (t_1 - \theta) \cdots (t_s - \theta)$.

\subsubsection{Fitting ideal of \texorpdfstring{$\lie(E_\alpha)(R_s/PR_s)$}{Lie(E)}} \label{sec-fittinglie}

Let us recall some facts about hyperdifferential operators.
For more details, we refer the reader to \cite{Con00}.

Let $j \geq 0$ be an integer. The $j^{th}$ \emph{hyperdifferential operator} $D_j$ is the $k_s$-linear endomorphism of $R_s$ given by $D_j(\theta^k) = \binom{k}{j} \theta^{k-j}$ for $k\geq 0$.
For any $f,g \in R_s$, we have the Leibnitz rule
\[ D_j(fg) = \sum\limits_{k = 0}^j D_k(f) D_{j - k}(g). \]

\begin{lem} \label{lem-partialhyperderi}
 For any $a \in R_s$, we have 
 \[ \partial(a)\begin{pmatrix} 0 \\ \vdots \\ 0 \\ 1 \end{pmatrix} = \begin{pmatrix} D_{n - 1}(a) \\ \vdots \\ D_1(a) \\ a \end{pmatrix}. \]
\end{lem}
\begin{proof}
 By linearity, it suffices to prove the equality for $a = \theta^k$, $k \in \mathbb N$.
 The action of $\partial(\theta^k)$ is the left multiplication by 
 \[ \begin{pmatrix} 
      \theta & 1 & & \\ 
      & \ddots & \ddots & \\ 
      & & \ddots & 1 \\ 
      & & & \theta 
     \end{pmatrix}^k 
    = \left(\theta I_n + 
    \begin{pmatrix} 
      0 & 1 & & \\ 
      & \ddots & \ddots & \\ 
      & & \ddots & 1 \\ 
      & & & 0 
    \end{pmatrix}\right)^k
    = \sum\limits_{i = 0}^k \binom{k}{i} \theta^{k - i} 
    \begin{pmatrix} 
      0 & 1 & & \\ 
      & \ddots & \ddots & \\ 
      & & \ddots & 1 \\ 
      & & & 0 \end{pmatrix}^i, \]
 hence the result comes from the definition of hyperdifferential operators. 
\end{proof}

\begin{lem} \label{lem-partial0modp}
 Let $P$ be a prime of $A$ and $m$ a positive integer.
 Then $\partial(P^m)$ is zero modulo $P$ if and only if $m$ is greater than or equal to $n$.
\end{lem}
\begin{proof}
 By the previous lemma, it suffices to show that for any $k\geq 0$, the congruence $D_k(P^m) = 0 \mod P$ holds if and only if $m \geq k+1$.
 The case $k = 0$ being obvious, let us suppose the result for an integer $k$.
 By the Leibnitz rule, we have
 \[ \begin{array}{rcl} 
      D_{k + 1}(P^m) & = & \displaystyle \sum\limits_{i + j = k + 1} D_i(P^{m - 1}) D_j(P) \\
      & = & P D_{k + 1}(P^{m - 1}) + D_1(P) D_k(P^{m - 1}) + \cdots + D_{k + 1}(P) P^{m - 1},
    \end{array} \]
 which is zero modulo $P$ if $m \geq k+2$.
 Reciprocally, observe that
 \[ \begin{array}{rcl} 
      D_{k + 1}(P^{k + 1}) & = & P D_{k + 1}(P^k) + D_1(P) D_k(P^k) + D_2(P) D_{k - 1}(P^k) + \cdots + D_{k + 1}(P) P^k \\
      & = & D_1(P) D_k(P^k) \mod P
    \end{array} \]
 which is non zero modulo $P$ by hypothesis. 
\end{proof}

Thanks to this lemma, we can compute the first Fitting ideal.

\begin{prop} \label{prop-fittinglie}
 Let $P$ be a prime of $A$.
 The $R_s$-module $\lie(E_\alpha)(R_s/PR_s)$ is isomorphic to $R_s/P^nR_s$ and is generated by the residue class of ${}^t(0,\dots,0,1)$.
\end{prop}
\begin{proof}
 By definition, $\lie(E_\alpha)(R_s/PR_s)$ is the $k_s$-vector space $(R_s/PR_s)^n$ equipped with the $R_s$-module structure given by $\partial$.
 This $R_s$-module is finitely generated and, since $\partial(P^{q^n}) = P^{q^n} I_n$ by lemma \ref{lem-fang}, the polynomial $P^{q^n}$ annihilates it.
 Since $R_s$ is principal, by the structure theorem, there exists integers $e_1 \leq \cdots \leq e_m$ such that 
 \[ \lie(E_\alpha)(R_s/PR_s) \simeq \frac{R_s}{P^{e_1}R_s} \times \cdots \times \frac{R_s}{P^{e_m}R_s}. \]
 Since $\lie(E_\alpha)(R_s/PR_s)$ is a $k_s$-vector space of dimension $n \deg P$, we have $e_1 + \cdots + e_m = n$.
 But, by the previous lemma, the residue class of ${}^t(0,\dots,0,1)$ is not annihilated by $P^{n-1}$, hence $e_m \geq n$.
 Thus, $\lie(E_\alpha)(R_s/PR_s)$ is cyclic and generated by the residue class of this vector. 
\end{proof}

\subsubsection{Fitting ideal of \texorpdfstring{$E_\alpha(R_s/PR_s)$}{E}} \label{sec-fittingE}

Let $P$ be a prime of $A$ and denote its degree by $d$.
We consider $R := R_s/PR_s$ and $E_\alpha(R)$ the $R_s$-module $R^n$ where the action of $R_s$ is given by $\phi$, as defined at the beginning of section \ref{sec-nthcarlitz}.

For $i = 1,\dots,n$, we denote by $e_i \colon \csinf^n \rightarrow \csinf$ the projection on the $i^{th}$ coordinate.
By analogy with \cite{AndTha90}, we define the $R_s$-module
\[ W_n(R) := \left\{ w \in R((t^{-1}))/R[t] \mid \alpha \tau(w) = (t - \theta)^n w \mod R[t] \right\}, \]
where $\tau(w) = \sum \tau(r_i) t^i$ if $w = \sum r_i t^i \in R((t^{-1}))$.

\begin{prop} \label{prop-isoERWR}
 The map 
 \[ \begin{array}{cccc} 
    \psi \colon & E_\alpha(R) & \longrightarrow & R((t^{-1}))/R[t] \\
    & c & \longmapsto & \displaystyle -\sum\limits_{i = 1}^\infty e_1 \phi_{\theta^{i - 1}}(c) t^{-i} \\
   \end{array} \]
 induces an isomorphism of $R_s$-modules between $E_\alpha(R)$ and $W_n(R)$.
\end{prop}
\begin{proof}
 See proposition 1.5.1 of \cite{AndTha90}. 
\end{proof}

Observe that for any $c \in E_\alpha(R)$, we have $\psi(\phi_\theta(c)) = t \phi_\theta(c) \mod R[t]$.
Moreover, since it is a $k_s$-vector space of dimension $nd$, $W_n(R)$ is a finitely generated and torsion $k_s[t]$-module.

For $w \in W_n(R)$, applying $d-1$ times $\alpha \tau$ to the relation $\alpha \tau(w) = (t - \theta)^n w$, we get
\[ \alpha\tau(\alpha) \cdots \tau^{d - 1}(\alpha) \tau^d(w) = \prod\limits_{i = 0}^{d - 1} \left(t - \theta^{q^i}\right)^n w. \]
But $\tau^d(w) = w$ in $W_n(R)$ and $\prod\limits_{i = 0}^{d - 1} (t - \theta^{q^i}) = P(t) \mod R[t]$ where $P(t)$ denotes the polynomial in $t$ obtained substituting $t$ form $\theta$ in $P$.
Thus we obtain
\begin{equation} \label{eqn-polyzerownr}
P^n(t) - \alpha\tau(\alpha) \cdots \tau^{d - 1}(\alpha) = 0 \ \text{in} \ W_n(R).
\end{equation}
Since we have the isomorphism
\[ \frac{R_s}{PR_s} \simeq \frac{A}{PA}\otimes_{\fq} k_s, \]
for any $x \in R_s$, there exists a unique $y \in k_s$ such that $x \tau(x) \cdots \tau^{d - 1}(x) = y \mod PR_s$.
We denote by $\rho_\alpha(P)$ the element of $k_s$ such that $\rho_\alpha(P) = \alpha \tau(\alpha) \cdots \tau^{d - 1}(\alpha) \mod PR_s$.
Note that, since $P$ is prime, $\rho_\alpha(P) = 0 \mod P$ if and only if $P$ divides $\alpha$ in $R_s$.
Then, by \eqref{eqn-polyzerownr}, we deduce that $W_n(R)$ is annihilated by $P^n(t) - \rho_\alpha(P)$, or equivalently 
\begin{equation} \label{eqn-ERinker}
 E_\alpha(R) \subseteq \ker\phi_{P^n - \rho_\alpha(P)} = \left\{ x \in R^n \mid \phi_{P^n - \rho_\alpha(P)}(x) = 0 \right\}.
\end{equation}

\begin{lem} \label{lem-dimatorsionWR}
 For any $a \in k_s[t]$ prime to $P(t) := P_{\mid_{\theta = t}}$, the $k_s$-vector space $W_n(R)[a]$ of $a$-torsion points of $W_n(R)$ is of dimension at most $\deg_t a$.
\end{lem}
\begin{proof}
 By definition, we have 
 \[ W_n(R)[a] = \left\{ w \in \frac{1}{a}R[t]/R[t] \mid \alpha \tau(w) = (t-\theta)^n w \mod R[t] \right\} \subseteq R((t^{-1}))/R[t]. \]
 Let $w \in W_n(R)[a]$.
 Since the $t^i/a$ for $i \in \{0,\dots,\deg a-1\}$ form an $R$-basis of $\frac{1}{a}R[t]/R[t]$, we can write
 \[ w = \sum\limits_{i = 0}^{\deg a - 1} \lambda_i \frac{t^i}{a}, \]
 where the $\lambda_i$ are in $R$.
 Using the binomial formula and writing $t^j/a$ for $j \geq \deg a$ in the above basis, the functional equation verified by $w$ becomes
 \[ \sum\limits_{i = 0}^{\deg a - 1} \alpha \tau(\lambda_i) \frac{t^i}{a} = \sum\limits_{i = 0}^{\deg a - 1} \sum\limits_j b_{i,j} \lambda_j \frac{t^i}{a}, \]
 where the $b_{i,j}$ are in $R$.
 Identifying the two sides, we obtain $\tau(\Lambda) = B \Lambda$ where $\Lambda$ is the vector ${}^t(\lambda_0,\dots,\lambda_{\deg a - 1})$ and $B$ is the matrix of $M_{\deg a}(R)$ with coefficients $b_{i,j}/\alpha$.
 
 But the $k_s$-vector space $V := \left\{ X \in R^{\deg a} \mid \tau(X) = BX \right\}$ is of dimension at most $\deg a$.
 Indeed, observe that, if $v_1,\dots,v_m$ are vectors of $R^{\deg a}$ such that $\tau(v_i) = B v_i$ for all~$i \in \left\{ 1,\dots,m \right\}$, linearly independent over $R$, there are also linearly independent over $R^\tau = k_s$ (by induction on $m$, see \cite[lemma 1.7]{vdPS}).  
\end{proof}

%
%

\begin{prop} \label{prop-fittingE}
 Let $P$ be a prime of $A$.
 We have the isomorphism of $R_s$-modules
 \[ E_\alpha(R) \simeq \frac{R_s}{(P^n - \rho_\alpha(P))R_s}. \]
\end{prop}
\begin{proof}
 Observe that if $P$ divides $\alpha$, we have $\rho_\alpha(P) = 0$ and the isomorphism of $R_s$-modules $\lie(E_\alpha)(R) \simeq E_\alpha(R)$.
 Then, the result is the same as in proposition \ref{prop-fittinglie}.
 
 Hence, let us suppose that $\alpha$ and $P$ are coprime.
 The $k_s$-vector space $E_\alpha(R)$ is of dimension~$nd$.
 We deduce from lemma \ref{lem-dimatorsionWR} that $E_\alpha(R)$ is a cyclic $R_s$-module, \ie
 \[ E_\alpha(R) \simeq \frac{R_s}{fR_s}, \]
 for some monic element $f$ of $R_s$ of degree $nd$.
 On the other hand, by the inclusion \eqref{eqn-ERinker}, $E_\alpha(R)$ is annihilated by $P^n - \rho_\alpha(P)$ thus $f$ divides $P^n - \rho_\alpha(P)$.
 Since these two polynomials are monic and have the same degree, they are equal. 
\end{proof}

\subsubsection{\texorpdfstring{$L$}{L}-values} \label{sec-cfPellarinL}

Let $a$ be a monic polynomial of $A$ and $a = P_1^{e_1} \cdots P_r^{e_r}$ be its decomposition into a product of primes.
Then, we define 
\[ \rho_\alpha(a) := \prod\limits_{i=1}^r \rho_\alpha(P_i)^{e_i}. \]
By propositions \ref{prop-fittinglie} and \ref{prop-fittingE}, we get 
\[ L(E_\alpha/R_s) = \prod\limits_{\substack{P \in A \\ \text{prime}}} \frac{[\lie(E_\alpha)(R_s/PR_s)]_{R_s}}{[E_\alpha(R_s/PR_s)]_{R_s}}
 = \prod\limits_{\substack{P\in A \\ \text{prime}}} \frac{P^n}{P^n-\rho_\alpha(P)} 
 = \sum\limits_{a\in A_+} \frac{\rho_\alpha(a)}{a^n} \in \ksinf. \]
 
As in \cite[section 4.1]{APTR14}, observe that for any prime $P$ of $A$, $\rho_\alpha(P)$ is the resultant of $P$ and $\alpha$ seen as polynomials in $\theta$.
In particular, if $\alpha = (t_1-\theta) \cdots (t_s-\theta)$, we obtain $\rho_\alpha(P) = P(t_1) \cdots P(t_s)$.
Thus, by theorem \ref{thm-cnf}, we get a class formula for $L$-values introduced in \cite{Pel12}: 
\[ L(\chi_{t_1} \cdots \chi_{t_s},n) = \sum\limits_{a \in A_+} \frac{\chi_{t_1}(a) \cdots \chi_{t_s}(a)}{a^n} 
= [\lie(E_\alpha)(R_s):\expei(E_\alpha(R_s))]_{R_s} [H(E_\alpha/R_s)]_{R_s}, \]
where $\chi_{t_i}\colon A \rightarrow \fq[t_1,\dots,t_s]$ are the ring homomorphisms defined respectively by $\chi_{t_i}(\theta) = t_i$.

\subsection{Goss abelian \texorpdfstring{$L$}{L}-series} \label{sec-gossL}

This section is inspired by \cite{AngTae12}.

Let $a \in A_+$ be squarefree and $L$ be the cyclotomic field associated with $a$, \ie the finite extension of $K$ generated by the $a$-torsion of the Carlitz module.
We denote by $\Delta_a$ the Galois group of this extension, it is isomorphic to $(A/aA)^\times$.

Note that $A[\Delta_a] = \prod\limits_i F_i[\theta]$ for some finite extensions $F_i$ of $\fq$.
In particular, $A[\Delta_a]$ is a principal ideal domain and Fitting ideals are defined as usual.
If $M$ is a finite $A[\Delta_a]$-module, we denote by $[M]_{A[\Delta_a]}$ the unique generator $f$ of $\fitt_{A[\Delta_a]} M$ such that each component $f_i \in F_i[\theta]$ of $f$ is monic.

We denote by $\widehat\Delta_a$ the group of characters of $\Delta_a$, \ie $\widehat\Delta_a = \hom(\Delta_a,\overline{\fq}^\times)$.
For $\chi \in \widehat\Delta_a$, we denote by $\fq(\chi)$ the finite extension of $\fq$ generated by the values of $\chi$ and we set
\[ e_\chi := \frac{1}{\#\Delta_a}\sum\limits_{\sigma \in \Delta_a} \chi^{-1}(\sigma) \sigma \in \fq(\chi)[\Delta_a]. \]
Then $e_\chi$ is idempotent and $\sigma e_\chi = \chi(\sigma) e_\chi$ for every $\sigma \in \Delta_a$.

Let $F$ be the finite extension of $\fq$ generated by the values of all characters, \ie $F$ is the compositum of all $\fq(\chi)$ for $\chi \in \widehat\Delta_a$.
If $M$ is an $A[\Delta_a]$-module, we have the decomposition into $\chi$-components
\[ F \otimes_{\fq} M = \bigoplus\limits_{\chi \in \widehat\Delta_a} e_\chi \left(F \otimes_{\fq} M\right). \]

Let $V$ be a free $\kinf[\Delta_a]$-module of rank $n$.
A sub-$A[\Delta_a]$-module $M$ of $V$ is a \emph{lattice} of $V$ if $M$ is free of rank one and $\kinf[\Delta_a] \cdot M = V$.
Let $M$ be a lattice of $V$ and $\chi \in \widehat\Delta_a$.
Then $M(\chi) := e_\chi \left( \fq(\chi) \otimes_{\fq} M \right)$ is a free $A(\chi)$-module of rank $n$, discrete in $V(\chi) := e_\chi \left( \fq(\chi) \otimes_{\fq} V \right)$, where $A(\chi) := \fq(\chi) \otimes_{\fq} A$.
Now let $M_1$ and $M_2$ be two lattices of $V$.
For each $\chi \in \widehat\Delta_a$, there exists $\sigma_\chi \in \gl(V(\chi))$ such that $\sigma_\chi(M_1(\chi)) = M_2(\chi)$.
Then, we define $[M_1(\chi) : M_2(\chi)]_{A(\chi)}$ to be the unique monic representative of $\det \sigma_\chi$ in  $\kinf(\chi) := \fq(\chi) \otimes_{\fq} \kinf$.
Finally, we set
\[ [M_1 : M_2]_{A[\Delta_a]} := \sum\limits_{\chi \in \widehat\Delta_a} [M_1(\chi) : M_2(\chi)]_{A(\chi)} e_\chi \in \kinf[\Delta_a]^\times. \]

\subsubsection{Gauss-Thakur sums} \label{sec-GTsums}

We review some  basic facts on Gauss-Thakur sums, introduced in \cite{Tha88} and generalized in \cite{AngPel14}.

We begin with the case of only one prime.
Let $P$ be a prime of $A$ of degree $d$ and $\zeta_P \in \overline{\fq}$ such that $P(\zeta_P) = 0$.
We denote by $\Lambda_P$ the $P$-torsion of the Carlitz module and let $\lambda_P$ be a non zero element of $\Lambda_P$.
We consider the cyclotomic extension $K_P := K(\Lambda_P) = K(\lambda_P)$ and we denote its Galois group by $\Delta_P$.
We have $\Delta_P \simeq (A/PA)^\times$.
More precisely, if $b \in (A/PA)^\times$, the corresponding element $\sigma_b \in \Delta_P$ is uniquely determined by $\sigma_b(\lambda_P) = C_b(\lambda_P)$.
We denote by $\mathcal O_{K_P}$ the integral closure of $A$ in $K_P$.
We have $\mathcal O_{K_P} = A[\lambda_P]$.

We define the \emph{Teichm\"uller character} 
\[ \omega_P \colon \begin{array}{ccc} \Delta_P & \longrightarrow & \mathbb F_{q^d}^* \\
                    \sigma_b & \longmapsto & b(\zeta_P)
                   \end{array}, \]
where $\sigma_b$ is the unique element of $\Delta_P$ such that $\sigma_b(\lambda_P) = C_b(\lambda_P)$.
Let $\chi \in \widehat\Delta_P$.
Since the Teichm\"uller character generates $\widehat\Delta_P$, there exists $j \in \{0, \dots, q^d - 2\}$ such that $\chi = \omega_P^j$.
We expand $j = j_0 + j_1 q + \cdots + j_{d - 1} q^{d - 1}$ in base $q$ ($j_0,\dots,j_{d - 1} \in \{0,\dots,q - 1\}$).
Then, the \emph{Gauss-Thakur sum} (see \cite{Tha88}) associated with $\chi$ is defined as
\[ g(\chi) := \prod\limits_{i = 0}^{d - 1} \left( -\sum\limits_{\delta \in \Delta_P} \omega_P^{-q^i}(\delta) \delta(\lambda_P) \right)^{j_i} \in \fq(\chi) \otimes_{\fq} \mathcal O_{K_P}. \]
We compute the action of $\tau = 1 \otimes \tau$ on these Gauss-Thakur sums (see \cite[proof of Theorem~II]{Tha88}).
Let $1 \leq j \leq d - 1$.
Since by the Carlitz action $\sigma_\theta \sigma_b(\lambda_P) = \theta \sigma_b(\lambda_P) + \tau\left( \sigma_b(\lambda_P) \right)$, we have
\[ \tau\left( g(\omega_P^{q^j}) \right) = -\sum\limits_{\sigma_b \in \Delta_P} \omega_P^{q^j}(\sigma_b) \left( \sigma_b \sigma_\theta(\lambda_P) - \theta \sigma_b(\lambda_P) \right) \]
Then, by substitution, we get 
\begin{equation} \label{eqn-taugqj}
 \tau\left( g(\omega_P^{q^j}) \right) = \left( \zeta_P^{q^j} - \theta \right) g(\omega_P^{q^j}).
\end{equation}

Now, we return to the general case.
Since $a$ is squarefree, we can write $a = P_1 \cdots P_r$ with $P_1,\dots,P_r$ distinct primes of respective degrees $d_1,\dots,d_r$.
Since $\widehat\Delta_a \simeq \widehat\Delta_{P_1} \times \cdots \times \widehat\Delta_{P_r}$, for every character $\chi \in \widehat\Delta_a$, we have
\begin{equation} \label{eqn-decompochi}
 \chi = \omega_{P_1}^{N_1} \cdots \omega_{P_r}^{N_r},
\end{equation}
for some integers $0 \leq N_i \leq q^{d_i} - 2$ and where $\omega_{P_i}$ is the Teichm\"uller character associated with~$P_i$.
The product $f_\chi := \prod\limits_{N_i \neq 0} P_i$ is the \emph{conductor} of $\chi$.
Then, the Gauss-Thakur sum (see \cite[section~2.3]{AngPel14}) associated with $\chi$ is defined as
\[ g(\chi) := \prod\limits_{i = 1}^r g(\omega_{P_i}^{N_i}) \in F \otimes_{\fq} \ol, \]
or equivalently
\[ g(\chi) = \prod\limits_{i = 1}^r \prod\limits_{j = 0}^{d_i - 1} g(\omega_{P_i}^{q^j})^{N_{i,j}}, \]
where the $N_{i,j}$ are the $q$-adic digits of $N_i$.
By equality \eqref{eqn-taugqj}, we obtain
\begin{equation} \label{eqn-taugchi}
 \tau\left( g(\chi) \right) = \underbrace{\prod\limits_{i = 1}^r \prod\limits_{j = 0}^{d_i - 1} \left( \zeta_{P_i}^{q^j} - \theta \right)^{N_{i,j}}}_{\alpha(\chi)} g(\chi).
\end{equation}

\begin{lem} \label{lem-eta}
 The ring $\ol$ is a free $A[\Delta_a]$-module of rank one generated by $\eta_a := \sum\limits_{\chi \in \widehat\Delta_a} g(\chi)$.
\end{lem}
\begin{proof}
 See lemma 16 of \cite{AngPel14}. 
\end{proof}

\subsubsection{The Frobenius action on the \texorpdfstring{$\chi$}{chi}-components} \label{sec-frobaction}

Recall that $L$ is the extension of $K$ generated by the $a$-torsion of the Carlitz module.
Let $\linf := L \otimes_K \kinf$ on which $\tau$ acts diagonally and $\Delta_a$ acts on $L$.
As in section \ref{sec-andmod}, we have a morphism of $A[\Delta_a]$-modules
\[ \expcn \colon \lie(\cn)(\linf) \longrightarrow \cn(\linf). \]
Let $\chi \in \widehat\Delta_a$.
We get an induced map
\[ \expcn \colon e_\chi \left(\lie(\cn)(\fq(\chi) \otimes_{\fq} \linf) \right) \longrightarrow \cn \left( e_\chi(\fq(\chi) \otimes_{\fq} \linf) \right), \]
where the action of $\tau$ on $\fq(\chi) \otimes_{\fq} \linf$ is on the second component.
But, by lemma \ref{lem-eta}, we have
\[ e_\chi(\fq(\chi) \otimes_{\fq} \linf) = g(\chi) \kinf(\chi), \]
where $\kinf(\chi) := \fq(\chi) \otimes_{\fq} \kinf$.

We have the obvious isomorphism of modules over $A(\chi) := \fq(\chi) \otimes_{\fq} A$ 
\[ g(\chi) \kinf(\chi) \stackrel{\sim}{\longrightarrow} \kinf(\chi), \]
where the action on the right hand side is denoted by $\tilde\tau$ and given by $\tilde\tau(f) = \alpha(\chi) (1 \otimes \tau) (f)$ for any $f \in \kinf(\chi)$, where $\alpha(\chi)$ is defined by equality \ref{eqn-taugchi}.
In particular, this isomorphism maps $\cnt$ into $\partial_\theta + N_1 \tilde\tau = \partial_\theta + N_{\alpha(\chi)} \tau$ with notation of section \ref{sec-nthcarlitz} and $\expcn$ into $\expachi$.
Thus, by lemma \ref{lem-eta}, we have the isomorphism of $A(\chi)$-modules
\[ e_\chi\left( \fq(\chi) \otimes_{\fq} H(\cn/ \ol) \right) \simeq \frac{E_{\alpha(\chi)}(\kinf(\chi))}{\expachi\left( \lie(E_{\alpha(\chi)})(\kinf(\chi)) \right) + E_{\alpha(\chi)}(A(\chi))}. \]
We denote the right hand side by $H\left( E_{\alpha(\chi)} / A(\chi) \right)$.
Note that we have also
\[ e_\chi \left( \fq(\chi) \otimes_{\fq} \expcni(\cn(\ol)) \right) = \expachii\left( E_{\alpha(\chi)}(A(\chi)) \right). \]

\subsubsection{\texorpdfstring{$L$}{L}-values} \label{sec-lvalues}

Let $\chi \in \widehat\Delta_a$ and denote its conductor by $f_\chi$.
Recall that the special value at $n \geq 1$ of Goss $L$-series (see \cite[chapter 8]{Gos}) associated with $\chi$ is defined by
\[ L(n,\chi) := \sum\limits_{b \in A_+} \frac{\chi(\sigma_b)}{b^n} \in \kinf(\chi), \]
where the sum runs over the elements $b \in A_+$ relatively prime to $f_\chi$.
If $b \in A_+$ and $f_\chi$ are not coprime, we set $\chi(\sigma_b) = 0$.
Then, define the Goss abelian $L$-series
\[ L(n,\Delta_a) := \sum\limits_{\chi \in \widehat\Delta_a} L(n,\chi) e_\chi \in \kinf[\Delta_a]^\times. \]

\begin{lem} \label{lem-LnDeltaprod}
 The infinite product 
 \[ \prod\limits_{\substack{P\in A \\ \text{prime}}} \frac{\left[ \lie(\cn)(\ol/P\ol) \right]_{A[\Delta_a]}}{\left[ \cn(\ol/P\ol) \right]_{A[\Delta_a]}} \]
 converges in $\kinf[\Delta_a]$ to $L(n,\Delta_a)$.
\end{lem}
\begin{proof}
 On the one hand, for all $\chi \in \widehat\Delta_a$, we have
 \[ L(n,\chi) = \prod\limits_{\substack{P\in A \\ \text{prime}}} \left( 1 - \frac{\chi(\sigma_P)}{P^n} \right)^{-1}, \]
 where $\chi(\sigma_P) = 0$ if $P$ divides $f_\chi$.
 On the other hand, let $\chi \in \widehat\Delta_a$.
 We write $\chi = \omega_{P_1}^{N_1} \cdots \omega_{P_r}^{N_r}$ as in equality \eqref{eqn-decompochi} and denote by $N_{i,j}$ the $q$-adic digits of $N_i$.
 Then, as in section \ref{sec-fittingE}, we can prove that
 \[ \begin{array}{rcl} \left[ E_{\alpha(\chi)}(A(\chi)/PA(\chi)) \right]_{A(\chi)} 
     & = & \displaystyle P^n - \prod\limits_{i = 1}^r \prod\limits_{j = 0}^{d_i - 1} P\left( \zeta_{P_i}^{q^j} \right)^{N_{i,j}} \\
     & = & \displaystyle P^n - \prod\limits_{i = 1}^r P(\zeta_{P_i})^{N_i} \\
     & = & P^n - \chi(\sigma_P).
    \end{array} \]
 Thus, we obtain
 \[ L(n,\chi) = \prod\limits_{\substack{P\in A \\ \text{prime}}} \frac{ \left[ \lie(E_{\alpha(\chi)})(A(\chi)/PA(\chi)) \right]_{A(\chi)} }{ \left[ E_{\alpha(\chi)}(A(\chi)/PA(\chi)) \right]_{A(\chi)} } \]
 Hence, we get the result by the discussion of section \ref{sec-frobaction} and definition of $L(n,\Delta_a)$. 
\end{proof}

Finally, we obtain a generalization of theorem A of \cite{AngTae12}:

\begin{thm} \label{thm-fcequi}
 Let $a \in A_+$ be squarefree and denote by $L$ the extension of $K$ generated by the $a$-torsion of the Carlitz module.
 In $\kinf[\Delta_a]$, we have
 \[ L(n,\Delta_a) = \left[ \lie(\cn)(\ol) : \expcni(\cn(\ol)) \right]_{A[\Delta_a]} \left[ H(\cn / \ol) \right]_{A[\Delta_a]}. \]
\end{thm}
\begin{proof}
 By the previous lemma, $L(n,\Delta_a)$ is expressed in terms of Anderson module and Fitting.
 Then, as in proposition \ref{prop-traceformula}, we express $L(n,\Delta_a)$ as a determinant.
 The proof is similar but we deal with the $\chi$-components $e_\chi(\fq(\chi) \otimes_{\fq} \ol)$ for all $\chi \in \widehat\Delta_a$.
 Then, since $A[\Delta_a]$ is principal, we conclude as in section \ref{sec-endproof}.
 We refer to \cite[paragraph 6.4]{AngTae12} for more details.
 \end{proof}

\bigskip

\bibliographystyle{amsplain}
\bibliography{biblio}
\addcontentsline{toc}{chapter}{References}

\bigskip

\noindent LMNO, CNRS UMR 6139, Universit\'e de Caen, 14032 Caen cedex, France
\newline \noindent E-mail address : \nolinkurl{florent.demeslay@unicaen.fr}

\end{document}